\documentclass[11pt,a4paper]{amsart}
\usepackage[T1]{fontenc}
\usepackage{lmodern}
\usepackage{amsmath,amssymb,amsthm,mathtools}
\usepackage{microtype}
\usepackage[margin=30mm]{geometry}
\usepackage[colorlinks=true,linkcolor=blue,citecolor=blue,urlcolor=blue]{hyperref}
\usepackage[alphabetic]{amsrefs}
\numberwithin{equation}{section}
\allowdisplaybreaks[2]
\newtheorem{theorem}{Theorem}[section]
\newtheorem{prop}[theorem]{Proposition}
\newtheorem{lemma}[theorem]{Lemma}

\newtheorem{definition}[theorem]{Definition}
\newtheorem{cor}[theorem]{Corollary}
\newtheorem{example}[theorem]{Example}

\newcommand{\R}{\mathbb R}
\newcommand{\dd}{\,\mathrm{d}}
\DeclareMathOperator{\Vol}{Vol}

\DeclareMathOperator{\AVR}{AVR}
\DeclareMathOperator{\rec}{rec}
\DeclareMathOperator{\dist}{dist}

\title[Curvature integral on convex hypersurfaces]
{Curvature integral on convex hypersurfaces}
\author{Qixuan Hu}
\address{Department of Mathematics and Computer Sciences,
Shantou University, Guangdong, P. R. China}
\email{qxhu@stu.edu.cn}
\thanks{The first author was supported by SRIG under Grant NTF25026T}
\author{Shuai Zhang}
\address{Department of Mathematical Sciences, Tsinghua University, Beijing, P. R. China}
\email{zhangshu22@mails.tsinghua.edu.cn}
\date{\today}

\begin{document}
\begin{abstract}
Let $M^n$, $n\ge3$, be the smooth, connected, complete, noncompact
boundary of a closed convex set with nonempty interior in $\R^{n+1}$.
We prove that its normalized scalar-curvature integral at infinity is
at most $4\pi\omega_{n-2}(1-\AVR(M))$, except when, up to an ambient
rigid motion, $M$ is the product of a smooth closed convex surface
in $\R^3$ and $\R^{n-2}$. These exceptional products have zero
asymptotic volume ratio and limit $8\pi\omega_{n-2}$.
We also characterize the equality case for the $4\pi\omega_{n-2}(1-\AVR(M))$ bound.
The proof uses exterior tube volumes together with the cone
Steiner and Gauss--Bonnet formulas. An elementary comparison of
intrinsic and Euclidean radial distances proves existence of the
intrinsic limit and identifies its exact value from the cone
coefficients. The same conclusions extend to  orientable convex submanifolds of arbitrary codimension
\end{abstract}
\maketitle

\section{Introduction}

The study of curvature integrals on complete noncompact manifolds has its origin in the classical theory of total curvature of surfaces.  Let $(\Sigma,h)$ be a complete noncompact Riemannian surface and let $K_h$ denote its Gaussian curvature.  Under the finite-total-curvature assumption, the Cohn--Vossen inequality gives
\[
    \int_{\Sigma} K_h\,dA_h
    \le 2\pi \chi(\Sigma),
\]
see \cite{CohnVossen1935}.  A more precise relation between total curvature and the geometry at infinity was obtained by Hartman \cite{Hartman1964}.  In terms of the scalar curvature $R_h=2K_h$, one has
\[
    \int_{\Sigma} R_h\,dA_h
    =
    4\pi\bigl(\chi(\Sigma)-\mathrm{AVR}(\Sigma,h)\bigr),
\]
under the corresponding integrability assumptions, where
\[
    \mathrm{AVR}(M,g)
    :=
    \lim_{r\to\infty}
    \frac{\operatorname{Vol}_g(B_p(r))}
         {\omega_n r^n}
\]
denotes the asymptotic volume ratio.  Thus, in dimension two, the total curvature records topological information and the asymptotic loss of Euclidean volume growth.

This picture naturally suggests higher-dimensional analogues.  Indeed, the quantity
\[
    r^{2-n}\int_{B_p(r)} R_g\,dV_g
\]
is invariant under a rescaling of the metric and the radius. Motivated by the two-dimensional total-curvature theory, Yau asked in \cite[Problem 9]{Yau1992} for higher-dimensional estimates of normalized integrals of the elementary symmetric functions of the Ricci tensor. In particular, on a complete noncompact manifold with nonnegative Ricci curvature, we wonder if
\[
    \limsup_{r\to\infty}
    r^{2-n}\int_{B_p(r)}R_g\,dV_g
    <\infty.
\]

In dimension three, Xu \cite{Xu2020} proved a weighted scalar-curvature integral estimate on nonparabolic manifolds with nonnegative Ricci curvature, formulated in terms of the Green function.  For manifolds with a pole, an exact unweighted formula was subsequently obtained in \cite{Xu2024}:
\[
    \lim_{r\to\infty}
    \frac{1}{r}
    \int_{B_p(r)}R_g\,dV_g
    =
    8\pi\bigl(1-\mathrm{AVR}(M,g)\bigr).
\]
In dimension three, the sharp asymptotic upper bound $8\pi$ were obtained by Munteanu and Wang \cite{MunteanuWang} when the scalar curvature is bounded above and below by positive constants. Cohn--Vossen-type results under additional geometric structures or decay assumptions were established by Deng \cite{Deng2026} and Ma \cite{Ma2026}.

On the other hand, the conjecture under the sole assumption $\operatorname{Ric}\ge0$ without further restrictions is false: counterexamples with divergent normalized scalar-curvature integral were recently constructed in \cites{XuCollapse2026,HaoZhu2026,Cheng2026}.  These examples indicate that some additional geometric structure is necessary for a sharp higher-dimensional theory.

A substantially different picture emerges under nonnegative sectional curvature.  In recent work, Xu \cite{Xu2026Sharp} proved, among other results, that for every complete noncompact manifold with nonnegative sectional curvature the limit
\[
    L(M,g)
    :=
    \lim_{r\to\infty}
    r^{2-n}\int_{B_p(r)}R_g\,dV_g
\]
exists and is finite.  When the manifold has a pole, the sharp estimate
\[
    L(M,g)
    \le
    4\pi\omega_{n-2}
    \bigl(1-\mathrm{AVR}(M,g)\bigr)
\]
holds in every dimension.  For general nonnegatively curved manifolds, sharp estimates and rigidity results were established in dimensions $3\le n\le6$, with the geometry of the soul playing an essential role.
The first author \cite{Hu2026} also obtained exact scalar-curvature
limit formulas for nonnegatively curved three-manifolds.

The purpose of the present paper is to study this problem in the setting of convex hypersurfaces.  Let
\[
    M^n=\partial K\subset\mathbb{R}^{n+1},
    \qquad n\ge3,
\]
be the smooth, connected, complete and noncompact boundary of a closed convex set $K$ with nonempty interior.  If $A$ denotes the shape operator with respect to the outward unit normal, convexity gives $A\ge0$, and the Gauss equation yields
\[
    \sec_M\ge0,
    \qquad
    R=2\sigma_2(A)\ge0.
\]
Thus convex hypersurfaces form a subclass of complete manifolds with nonnegative sectional curvature.  The Euclidean embedding carries additional information that is invisible from the intrinsic metric alone.  Our main observation is that this extrinsic information can be used to convert the asymptotic curvature-integral problem into a problem in convex geometry.

For $p\in M$, set
\[
    L(M)
    :=
    \lim_{r\to\infty}
    r^{2-n}
    \int_{B_p^M(r)}R\,dV,
    \qquad
    V_M:=\mathrm{AVR}(M).
\]
We prove first that the above limit always exists, is finite, and is independent of the base point.  Moreover, except for one explicitly determined family of cylindrical products, we obtain the sharp estimate
\[
    L(M)
    \le
    4\pi\omega_{n-2}(1-V_M).
\]
The exceptional case occurs precisely when, after an ambient rigid motion,
\[
    K=D\times\mathbb{R}^{n-2},
    \qquad
    M=\partial D\times\mathbb{R}^{n-2},
\]
where $D\subset\mathbb{R}^3$ is a compact convex body with smooth boundary.  In this case
\[
    V_M=0,
    \qquad
    L(M)=8\pi\omega_{n-2}.
\]
In particular, the coefficient $4\pi\omega_{n-2}$, which in the general nonnegatively curved setting is naturally associated with the existence of a pole, persists for convex hypersurfaces in every dimension without any pole assumption, apart from this completely classified product obstruction.

We also characterize the equality case in the estimate.  Equality holds precisely when either $(M,g)$ is isometric to Euclidean space or, after an ambient rigid motion,
\[
    M=\partial D_0\times\mathbb{R}^k,
    \qquad
    n-3\le k\le n-2,
\]
where $D_0\subset\mathbb{R}^{n+1-k}$ is a closed noncompact convex set with nonempty interior containing no complete affine line.

The proof is based on the recession cone of $K$,
\[
    C=\operatorname{rec}(K),
\]
which captures the asymptotic geometry of the convex body.  After translating the base point to the origin and rescaling
\[
    K_r=r^{-1}K,
\]
the sets $K_r$ decrease to $C$ as $r\to\infty$.  For a closed convex set $Q\subset\mathbb{R}^{n+1}$, we consider the volume $T_Q(t)$ of the portion of its exterior $t$-tube whose nearest-point projection lies in the unit ball. We first calculate $T_{K_r}(t)$, which turns out to be a polynomial, with coefficients of $t$ and $t^3$ equal to
\[
    r^{-n}\operatorname{Vol}(M\cap B_r)
    \quad\text{and}\quad
    \frac16\,r^{2-n}
    \int_{M\cap B_r}R\,dV
\]
respectively. The convergence $K_r\to C$ then implies convergence of the corresponding tube volumes and hence of all their coefficients. $T_C(t)$ is then a polynomial with coefficients of $t$ and $t^3$ equal to $V_M$ and $L(M)$ respectively.

Now we want to calculate $T_C(t)$. When the recession cone is not a linear subspace, the Steiner formula for convex cones \cite{McCoyTropp} identifies these coefficients with the conic intrinsic volumes $v_j(C)$.  More precisely, the relevant coefficients give
\[
    V_M=2v_n(C),
    \qquad
    L(M)=8\pi\omega_{n-2}v_{n-2}(C).
\]
The conic Gauss--Bonnet relation then yields the exact deficit formula
\[
\begin{split}
    &4\pi\omega_{n-2}(1-V_M)-L(M)
    \\
    &\qquad =
    8\pi\omega_{n-2}
    \sum_{\substack{0\le j\le n-4\\ j\equiv n\;(\mathrm{mod}\,2)}}
    v_j(C)
    \ge0.
\end{split}
\]
This identity not only proves the sharp inequality but also reduces its rigidity to the vanishing pattern of the conic intrinsic volumes.  The latter is determined by the linear space
\[
    C\cap(-C),
\]
leading to the equality classification above.

The remaining case, where $C$ is a linear subspace, has a different geometric character.  Convexity gives a global splitting
\[
    K=D\times\mathbb{R}^k,
    \qquad
    M=\partial D\times\mathbb{R}^k,
\]
with $D$ compact.  A direct product computation shows that the normalized scalar-curvature integral vanishes unless $k=n-2$.  In the latter case $\partial D$ is a closed convex surface in $\mathbb{R}^3$, and the Gauss--Bonnet theorem gives
\[
    \int_{\partial D}R\,dV=8\pi,
\]
which is exactly responsible for the exceptional value
$L(M)=8\pi\omega_{n-2}$.

Finally, by combining the hypersurface result with Jonker's structure theorem for complete convex immersions \cite{Jonker}, we extend the same conclusions to orientable convex submanifolds of arbitrary codimension: the nonflat case reduces to a convex hypersurface in an affine $(m+1)$-dimensional subspace, while the remaining alternative is intrinsically flat.

\section{Statement and notation}

Let $K\subset\R^{n+1}$ be a closed convex set with nonempty interior,
and suppose that $M=\partial K$ is smooth, connected, complete, and
noncompact. The metric on $M$ is induced from the ambient
Euclidean metric. 
For the outward unit normal $\nu$ of $M$, define the
shape operator $A(X)=D_X\nu$, which is nonnegative by convexity.
Write $\sigma_i(A)$ for the $i$-th
elementary symmetric polynomial of its eigenvalues, with $\sigma_0=1$.
The Gauss equation implies that the sectional curvature of $M$ is
nonnegative, and the scalar curvature $R$ on $M$ satisfies
\begin{equation}\label{eq:gauss}
 R=2\sigma_2(A)\ge0.
\end{equation}

Let $\omega_k$ be the volume of the unit ball in
$\R^k$, with $\omega_0=1$. 

For $p\in M$, distinguish the intrinsic and extrinsic balls by
\[
 B_p^M(r)=\{x\in M:d_M(p,x)<r\},\qquad
 E_p(r)=M\cap B_p^{\R^{n+1}}(r).
\]
Bishop--Gromov's volume comparison theorem gives
the basepoint-independent asymptotic volume ratio
\[
 V_M:=\AVR(M)=\lim_{r\to\infty}
 \frac{\Vol_M(B_p^M(r))}{\omega_n r^n}\in[0,1].
\]

\begin{theorem}\label{thm:main}
Under these hypotheses, for every $n\ge3$ the limit
\[
 L(M):=\lim_{r\to\infty}r^{2-n}\int_{B_p^M(r)}R\dd V
\]
exists, is finite, and is independent of $p$.
One has
\begin{equation}\label{eq:strong-bound}
 L(M)\le4\pi\omega_{n-2}(1-V_M),
\end{equation}
except when, after an ambient rigid motion,
\begin{equation}\label{eq:exception-product}
	K=D\times\R^{n-2},\qquad M=\partial D\times\R^{n-2},
\end{equation}
where $D\subset\R^3$ is a compact convex set with nonempty interior
and smooth boundary. 
In the exceptional case \eqref{eq:exception-product},
\begin{equation}\label{eq:exception-values}
 V_M=0,\qquad L(M)=8\pi\omega_{n-2}.
\end{equation}

Equality in \eqref{eq:strong-bound} holds precisely in one of the following cases:
\begin{enumerate}
\item
$(M,g)$ is isometric to $\mathbb{R}^n$.

\item
After an ambient rigid motion,
\[
M
=
\partial D_0\times\mathbb{R}^k
\subset
\mathbb{R}^{n+1},
\]
where $n-3\leq k\leq n-2$ and $D_0\subset\mathbb{R}^{n+1-k}$ is a closed noncompact convex set with nonempty interior that contains no complete affine line.
\end{enumerate}
\end{theorem}

Xu proves existence of the normalized intrinsic limit more generally
for complete noncompact manifolds of nonnegative sectional curvature
\cite[Theorem 4.1, equation (4.4); see also Theorem B]{Xu2026Sharp}.
Here we obtain existence directly: intrinsic and Euclidean radial
distances are asymptotically equal, so the exterior tube coefficients
give the intrinsic volume ratio and curvature integral exactly.
This also allows us to characterize equality in
\eqref{eq:strong-bound}. 

\section{Exterior tubes and their limits}

For a nonempty closed convex set $Q\subset\R^{n+1}$, let $\Pi_Q(y)$
denote the unique nearest point of $Q$ to $y$. For $t>0$, let
\begin{equation}\label{eq:tube-definition}
 U_Q(t)=\{y\in\R^{n+1}:0<\dist(y,Q)<t,\ |\Pi_Q(y)|<1\}
\end{equation}
be the part of the exterior tube whose nearest point
lies in the open unit ball.
Set
\begin{equation}
 T_Q(t)=\Vol_{n+1}(U_Q(t)).
\end{equation}
Note that $U_Q(t)\subset B_0^{\R^{n+1}}(1+t)$, so $T_Q(t)<\infty$.

Fix $p\in M$. by translation we may assume that $p=0$.
Write $B_r=B_0^{\R^{n+1}}(r)$ and set $K_r=r^{-1}K$.

\begin{lemma}\label{lem:tube-polynomial}
For every $r,t>0$,
\begin{equation}\label{eq:tube-polynomial}
 T_{K_r}(t)=\sum_{q=0}^{n}\frac{t^{q+1}}{q+1}
 r^{q-n}\int_{M\cap B_r}\sigma_q(A)\dd V.
\end{equation}
In particular, the coefficients of $t$ and $t^3$ are respectively
\begin{equation}\label{eq:two-coefficients}
 r^{-n}\Vol_M(M\cap B_r),\qquad
 \frac16r^{2-n}\int_{M\cap B_r}R\dd V.
\end{equation}
\end{lemma}

\begin{proof}
For fixed $r>0$, let $M_r=\partial K_r$ and let $\nu_r$ be its
outward unit normal. Define
\begin{align*}
    \Phi(x,s)=x+s\nu_r(x)
\end{align*}
for $(x,s)\in M_r\times(0,\infty)$. Then $\Phi$ is a bijection
from $M_r\times(0,\infty)$ to $\R^{n+1}\setminus K_r$, and
$\Pi_{K_r}(\Phi(x,s))=x$. Indeed, the supporting-plane inequality
$\langle\nu_r(x),z-x\rangle\le0$ for $z\in K_r$ shows that
$x$ is the nearest point to $x+s\nu_r(x)$. Conversely, every
point outside $K_r$ has a unique nearest point on $M_r$, and
its displacement is a positive multiple of the outward normal.

Its Jacobian is
\begin{align*}
    \mathrm{Jac}(\Phi)(x,s)=\det(I+sA_r(x))=\sum_{q=0}^{n}\sigma_q(A_r(x))s^q,
\end{align*}
where $A_r(x)$ is the shape operator of $M_r$ at $x$.
The factors $1+s\kappa_i(A_r)$ are positive. By the area formula,
\begin{align*}
 T_{K_r}(t)
 &=\int_0^t\int_{M_r\cap B_1}
   \mathrm{Jac}(\Phi)(x,s)\dd V_{M_r}(x)\dd s\\
 &=\sum_{q=0}^n\frac{t^{q+1}}{q+1}
   \int_{M_r\cap B_1}\sigma_q(A_r(x))\dd V_{M_r}(x)\\
 &=\sum_{q=0}^n\frac{t^{q+1}}{q+1}
   r^{q-n}\int_{M\cap B_r}\sigma_q(A(x))\dd V_M(x).
\end{align*}
Here $A_r(x/r)=rA(x)$ under the natural identification of tangent
spaces, and the area element scales by $r^{-n}$.
Equation~\eqref{eq:gauss} gives \eqref{eq:two-coefficients}.
The integrals are finite because $M\cap\overline B_r$ is a compact
part of a smooth hypersurface that is closed as a subset of
$\R^{n+1}$.
\end{proof}

\begin{definition}
    Let
    \begin{equation}\label{eq:rec-definition}
     \rec(K)=\{v\in \mathbb{R}^{n+1}:x+tv\in K\text{ for all }x\in K, t\ge0\}
    \end{equation}
    be the recession cone of $K$.
\end{definition}
\begin{lemma}\label{lem:recession}
    $\rec(K)=\bigcap_{r>0}K_r$.
\end{lemma}
\begin{proof}
    Since $0\in K$ and $K$ is convex, the sets $K_r$ decrease as $r$ increases. If $v\in\rec(K)$,
then $sv\in K$ for every $s>0$, so $v\in\bigcap_{r>0}K_r$.

Conversely, if $v\in\bigcap_{r>0}K_r$, then $sv\in K$ for
every $s>0$. For $x\in K$, $t\ge0$, and $s>t$, convexity gives
\[
 (1-t/s)x+tv=(1-t/s)x+(t/s)(sv)\in K.
\]
Letting $s\to\infty$ and using closedness gives $x+tv\in K$. Thus
$\bigcap_{r>0}K_r=\rec(K)$.
\end{proof}

\begin{lemma}\label{lem:tube-limit}
Denote $C=\rec(K)$. For every $y\in\R^{n+1}$,
\begin{equation}\label{eq:projection-tube-limit}
 \lim_{r\to\infty}\Pi_{K_r}(y)=\Pi_C(y),
\end{equation}
and for every $t>0$,
\begin{equation}
    \lim_{r\to\infty}T_{K_r}(t)= T_C(t).
\end{equation}
Furthermore, $T_C(t)=\sum_{m=1}^{n+1}c_m t^m$ is a polynomial,
and each $c_m$ is the limit of the coefficient of $t^m$ in
$T_{K_r}(t)$. In particular,
\begin{equation}\label{eq:extrinsic-limits}
 \lim_{r\to\infty}r^{-n}\Vol_M(M\cap B_r)=c_1,
 \qquad
 \lim_{r\to\infty}r^{2-n}\int_{M\cap B_r}R\dd V=6c_3.
\end{equation}
\end{lemma}

\begin{proof}
Fix $y\in\R^{n+1}$ and set $x_r=\Pi_{K_r}(y)$.
Since $0\in K_r$, we have $|y-x_r|\le|y|$, and hence
\begin{align*}
    |x_r|\leq |y-x_r|+|y|\leq 2|y|.
\end{align*}
That is, $\{x_r\}_{r>0}$ is bounded.

Along any sequence $r_i\to\infty$, choose a convergent subsequence
of $x_{r_i}$, and denote its limit by $x$. Relabel this subsequence
as $x_{r_i}$. For each fixed $s>0$, eventually $r_i\ge s$, so
$x_{r_i}\in K_s$. Closedness gives $x\in K_s$, hence $x\in C$.
By the definition of the nearest point and $C\subset K_{r_i}$,
\begin{align*}
    |y-x_{r_i}|=\dist(y,K_{r_i})\le\dist(y,C).
\end{align*}
Passing to the limit gives
\begin{align*}
    |y-x|\leq \dist(y,C).
\end{align*}

Since $x\in C$, equality must hold, so $x=\Pi_C(y)$.
Every convergent subsequence has this same limit; boundedness
therefore gives $\Pi_{K_r}(y)\to\Pi_C(y)$ as $r\to\infty$.
It also gives $\dist(y,K_r)\to\dist(y,C)$.

Since $C$ is invariant under positive dilations, $\lambda C=C$
for every $\lambda>0$. Hence
\[
 \dist(\lambda y,C)=\dist(\lambda y,\lambda C)
 =\lambda\dist(y,C).
\]
The point $\lambda\Pi_C(y)\in\lambda C$ attains this distance.
Uniqueness of the nearest point therefore gives
\[
 \Pi_C(\lambda y)=\Pi_{\lambda C}(\lambda y)=\lambda\Pi_C(y).
\]
It follows by polar coordinates that both sets
\[
 \{y:|\Pi_C(y)|=1\},\qquad
 \{y:\dist(y,C)=t\}
\]
have $(n+1)$-dimensional volume zero: on each ray from the origin,
each indicated positive level is attained at most once.
Outside these sets, if $y\notin C$, convergence of projections and
distances gives
$\mathbf1_{U_{K_r}(t)}(y)\to\mathbf1_{U_C(t)}(y)$.
If $y\in C$, then $y\in K_r$ for every $r$, so all these indicators
are zero. All indicators vanish outside $B_{1+t}$. The dominated
convergence theorem then proves $T_{K_r}(t)\to T_C(t)$.

By Lemma~\ref{lem:tube-polynomial}, $T_{K_r}$ is a polynomial
of degree at most $n+1$ for $t>0$ and extends to $t=0$ with
value zero. Set $t_0=0$ and choose distinct positive numbers
$t_1,\ldots,t_{n+1}$. The Lagrange polynomials
\[
 P_k(t)=\prod_{\substack{0\le j\le n+1\\j\ne k}}
 \frac{t-t_j}{t_k-t_j},\qquad 0\le k\le n+1,
\]
are fixed independently of $r$. Lagrange interpolation gives
for every $t\ge0$
\begin{align*}
 T_{K_r}(t)=\sum_{k=0}^{n+1}T_{K_r}(t_k)P_k(t)
 =\sum_{k=1}^{n+1}T_{K_r}(t_k)P_k(t).
\end{align*}
Letting $r\to\infty$ gives
$T_C(t)=\sum_{k=1}^{n+1}T_C(t_k)P_k(t)$ for $t\ge0$.
Thus $T_C$ is a polynomial, and each coefficient is the limit
of the corresponding coefficient of $T_{K_r}$. In particular,
\eqref{eq:two-coefficients} implies \eqref{eq:extrinsic-limits}.
\end{proof}

\begin{lemma}[Asymptotic comparison of radial distances]\label{lem:radial-distance}
For every $p\in M$,
\begin{equation}\label{eq:radial-distance}
 \frac{d_M(p,x)}{|x-p|}\longrightarrow1
 \qquad\text{as }|x-p|\longrightarrow\infty,
\end{equation}
uniformly in $x\in M$. Consequently, for every $\varepsilon>0$
and all sufficiently large $r$,
\begin{equation}\label{eq:radial-ball-comparison}
 E_p\!\left(\frac{r}{1+\varepsilon}\right)
 \subset B_p^M(r)\subset E_p(r).
\end{equation}
\end{lemma}

\begin{proof}
	By translation, assume that $p=0$, and set $\rho(x)=|x|$ on
	$M\setminus\{0\}$. We first show that
	\begin{equation}\label{eq:radial-gradient}
		|\nabla_M\rho(x)|\longrightarrow1
		\qquad\text{as }|x|\longrightarrow\infty,
	\end{equation}
	uniformly in $x\in M$.
	
	Let $x_j\in M$ satisfy $r_j=|x_j|\to\infty$, and put
	$u_j=x_j/r_j$. After passing to a subsequence, assume that
	$u_j\to u$. For each fixed $s>0$, we have
	$u_j\in K_{r_j}\subset K_s$ for all sufficiently large $j$.
	By closedness and Lemma~\ref{lem:recession},
	\[
	u\in\bigcap_{s>0}K_s=C.
	\]
	Let $\nu_j$ be the outward unit normal at $x_j$. Applying the
	supporting hyperplane inequality to $0\in K$ and $x_j+u\in K$
	gives
	\[
	\langle u_j,\nu_j\rangle\ge0,
	\qquad
	\langle u,\nu_j\rangle\le0.
	\]
	Consequently,
	\[
	0\le\langle u_j,\nu_j\rangle
	\le\langle u_j-u,\nu_j\rangle
	\le|u_j-u|\longrightarrow0,
	\]
	and hence
	\[
	|\nabla_M\rho(x_j)|^2
	=1-\langle u_j,\nu_j\rangle^2\longrightarrow1.
	\]
	A contradiction argument now proves the uniform convergence
	in \eqref{eq:radial-gradient}.
	
	Fix $0<\delta<1$. Choose $T>0$ such that
	$|\nabla_M\rho|\ge1-\delta$ whenever $\rho\ge T$. Starting from any $x\in M$ with $|x|>T$, consider
	the flow $\gamma$ of
	\[
	X=-\frac{\nabla_M\rho}{|\nabla_M\rho|^2},
	\qquad \gamma(0)=x.
	\]
	Along this flow,
	\[
	\frac{d}{ds}\rho(\gamma(s))=-1,
	\qquad |X|\le\frac1{1-\delta}.
	\]
	The vector field is smooth on a neighborhood of the compact
	annulus $M\cap\{T\le\rho\le|x|\}$. Thus the flow reaches a point
	$z\in M\cap\{\rho=T\}$ at time $|x|-T$, and
	\[
	d_M(x,z)\le\frac{|x|-T}{1-\delta}.
	\]
	Since $M$ is connected and $M\cap\{\rho\le T\}$ is compact,
	\[
	A_T:=\max_{y\in M,\ |y|\le T}d_M(0,y)<\infty.
	\]
	It follows that
	\[
	|x|\le d_M(0,x)
	\le A_T+\frac{|x|-T}{1-\delta}.
	\]
	Dividing by $|x|$, letting $|x|\to\infty$, and then letting
	$\delta\downarrow0$ proves \eqref{eq:radial-distance} uniformly.
	
	Finally, given $\varepsilon>0$, choose $S>0$ such that
	\[
	d_M(0,x)\le(1+\varepsilon)|x|
	\qquad\text{whenever }|x|>S.
	\]
	The distance to $0$ has a finite maximum on
	$M\cap\overline B_0^{\R^{n+1}}(S)$.
	For $r$ larger than this maximum, every point of
	$E_0(\frac{r}{1+\varepsilon})$ has intrinsic distance less than $r$, and therefore lies in $B_0^M(r)$.
	The reverse inclusion $B_0^M(r)\subset E_0(r)$ follows from
	$|x|\le d_M(0,x)$.
\end{proof}

\begin{cor}[Intrinsic limits]\label{cor:intrinsic-limits}
Let $c_1,\ldots,c_{n+1}$ be the coefficients of $T_C$ in
Lemma~\ref{lem:tube-limit}. Then, for every $p\in M$,
\begin{equation}\label{eq:intrinsic-coefficients}
 V_M=\frac{c_1}{\omega_n},\qquad
 L(M):=\lim_{r\to\infty}r^{2-n}
 \int_{B_p^M(r)}R\dd V=6c_3.
\end{equation}
In particular, the scalar-curvature limit exists, is finite,
and is independent of $p$.
\end{cor}

\begin{proof}
For fixed $\varepsilon>0$, apply Lemma~\ref{lem:tube-limit} and Lemma~\ref{lem:radial-distance} to
obtain
\[
 \frac{c_1}{(1+\varepsilon)^n}
 \le\liminf_{r\to\infty}r^{-n}\Vol_M(B_p^M(r))
 \le\limsup_{r\to\infty}r^{-n}\Vol_M(B_p^M(r))
 \le c_1.
\]
Letting $\varepsilon\downarrow0$ gives $V_M=c_1/\omega_n$.
Similarly, since $R\ge0$ and $n\ge3$,
\[
 \frac{6c_3}{(1+\varepsilon)^{n-2}}
 \le\liminf_{r\to\infty}r^{2-n}\int_{B_p^M(r)}R\dd V
 \le\limsup_{r\to\infty}r^{2-n}\int_{B_p^M(r)}R\dd V
 \le6c_3.
\]
Letting $\varepsilon\downarrow0$ proves the second equality in
\eqref{eq:intrinsic-coefficients}. The coefficients are finite
by Lemma~\ref{lem:tube-limit}. Finally, translating the basepoint
does not change $C=\rec(K)$, so it does not change the polynomial
$T_C$ or its coefficients. Hence $L(M)=6c_3$ is independent of $p$.
\end{proof}

\section{Recession cones that are not linear subspaces}\label{sec:nonlinear}

\begin{definition}
    A nonempty set $C\subset\R^{n+1}$ is a closed convex cone if it is closed and convex and
    \begin{align*}
        \lambda C= C
    \end{align*}
    for all $\lambda>0$. In particular, $0\in C$.
\end{definition}

In this section, we treat the case in which $C=\rec(K)$ is not a
linear subspace. 

We first recall
the master Steiner formula of McCoy--Tropp
\cite[Theorem~3.1 and Corollary~3.2, pp.~931--932]{McCoyTropp} in the form used
here. There are nonnegative numbers $v_0(C),\ldots,v_{n+1}(C)$,
with $\sum_{k=0}^{n+1}v_k(C)=1$, such that, for every bounded
Borel function $f:[0,\infty)^2\to\R$,
\begin{equation}\label{eq:master-steiner}
 \mathbb E f\bigl(|\Pi_C(G)|^2,\dist(G,C)^2\bigr)
 =\sum_{k=0}^{n+1}v_k(C)\,
 \mathbb E f\bigl(|G_k|^2,|H_{n+1-k}|^2\bigr).
\end{equation}
The numbers $v_k(C)$ are the conic intrinsic volumes of $C$
defined in \cite[Definition~2.2, p.~930]{McCoyTropp}; their nonnegativity
is also recorded in \cite[Fact~5.5(1), equation~(5.1), p.~250]{ALMT}.
Here $G$ is a standard Gaussian vector in $\R^{n+1}$; for each
$k$, $G_k$ and $H_{n+1-k}$ are independent standard Gaussian
vectors in $\R^k$ and $\R^{n+1-k}$, respectively. And $\mathbb E$ is the expectation. A zero-dimensional
Gaussian vector is understood to be $0$.

\begin{lemma}[Cone projection]\label{lem:cone-projection}
	Let $C\subset\R^d$ be a nonempty closed convex cone. Then, for
	every $y\in\R^d$,
	\[
	\langle \Pi_C(y),y-\Pi_C(y)\rangle=0,
	\qquad
	|\Pi_C(y)|^2+\dist(y,C)^2=|y|^2.
	\]
\end{lemma}

\begin{proof}
	Set $x=\Pi_C(y)$. Since $tx\in C$ for every $t\ge0$, the function
	$t\mapsto |y-tx|^2$ attains its minimum at $t=1$. Differentiating
	at $t=1$ gives
	\[
	0=-2\langle y-x,x\rangle.
	\]
	The second identity follows by expanding $|y|^2=|x+(y-x)|^2$
	and using $|y-x|=\dist(y,C)$.
\end{proof}

The conic intrinsic volumes satisfy the following Gauss--Bonnet
formula \cite[Fact~5.5(3), equation~(5.3), p.~251]{ALMT}.

\begin{lemma}[Conic Gauss--Bonnet formula]\label{lem:conic-gauss-bonnet}
	Let $C\subset\R^{n+1}$ be a nonempty closed convex cone
	that is not a linear subspace. Then
	\begin{equation}\label{eq:conic-gauss-bonnet}
		\sum_{\substack{0\le k\le n+1\\ k\ \mathrm{even}}}v_k(C)
		=
		\sum_{\substack{0\le k\le n+1\\ k\ \mathrm{odd}}}v_k(C)
		=\frac12.
	\end{equation}
\end{lemma}

\begin{lemma}[Cone coefficients]\label{lem:cone-coefficients}
Let $C\subset\R^{n+1}$ be a closed convex cone, where $n\ge3$. Suppose $T_C(t)=\sum_{m=1}^{n+1}c_m t^m$ for $t>0$, and let $v_k(C)$ be the conic intrinsic volumes. Then
\begin{equation}\label{eq:cone-coefficients}
 v_{n+1-m}(C)=\frac{c_m}{\omega_{n+1-m}\omega_m},
 \qquad 1\le m\le n+1.
\end{equation}
If $C$ is not a linear subspace, then
\begin{equation}\label{eq:odd-coefficients}
 \sum_{\substack{1\le m\le n+1\\m\ \mathrm{odd}}}
 \frac{c_m}{\omega_{n+1-m}\omega_m}=\frac12.
\end{equation}
\end{lemma}

\begin{proof}
Fix $t>0$. To compute $T_C(t)$, we apply the master Steiner
formula \eqref{eq:master-steiner} with the test function
\[
    f_t(a,b)=(2\pi)^{(n+1)/2}e^{(a+b)/2}\mathbf{1}_{\{a<1,\;0<b<t^2\}}.
\]
Since on the support of $f_t$ we have
\[
    0\leq a<1, \qquad 0<b<t^2,
\]
it follows that
\[
    0\leq f_t(a,b)\leq(2\pi)^{(n+1)/2}e^{(1+t^2)/2}.
\]
Thus $f_t$ is indeed a bounded Borel function, so it may be inserted
into the master Steiner formula.

We first calculate the left hand side of \eqref{eq:master-steiner}. For every $y\in\mathbb{R}^{n+1}$, Lemma \ref{lem:cone-projection} gives
\[
    |\Pi_C(y)|^2+\operatorname{dist}(y,C)^2=|y|^2.
\]

Since the density of the standard Gaussian vector $G$ is
\[
    (2\pi)^{-(n+1)/2}e^{-|y|^2/2},
\]
we obtain
\begin{align*}
    &\mathbb{E}f_t\!\left(|\Pi_C(G)|^2,\operatorname{dist}(G,C)^2\right)\\
    =&\int_{\mathbb{R}^{n+1}}(2\pi)^{-(n+1)/2}e^{-|y|^2/2}(2\pi)^{(n+1)/2}e^{\frac{1}{2}\left(|\Pi_C(y)|^2+\operatorname{dist}(y,C)^2\right)}\\
    &\hspace{4cm}\times\mathbf{1}_{\{|\Pi_C(y)|^2<1,\,0<\operatorname{dist}(y,C)^2<t^2\}}\,dy.
\end{align*}
Using
\[
    |\Pi_C(y)|^2+\operatorname{dist}(y,C)^2=|y|^2,
\]
the exponential factors cancel. Therefore
\begin{align}\label{eq:left of steiner formula}
    \mathbb{E}f_t\!\left(|\Pi_C(G)|^2,\operatorname{dist}(G,C)^2\right) =\int_{\mathbb{R}^{n+1}}\mathbf{1}_{\{|\Pi_C(y)|<1,\,0<\operatorname{dist}(y,C)<t\}}\,dy=T_C(t).
\end{align}

Now we turn to the right hand side of \eqref{eq:master-steiner}. Fix $0\leq k\leq n$. Since $G_k$ and $H_{n+1-k}$ are independent, their joint Gaussian density on
$\mathbb{R}^k\times\mathbb{R}^{n+1-k}$ is
\[
    (2\pi)^{-(n+1)/2}\exp\left(-\frac{|x|^2+|z|^2}{2}\right).
\]
Consequently,
\begin{align}\label{eq:right of steiner formula}
    &\mathbb{E}f_t\!\left(|G_k|^2,|H_{n+1-k}|^2\right)\\
    =&\int_{\mathbb{R}^k}\int_{\mathbb{R}^{n+1-k}}(2\pi)^{-(n+1)/2}e^{-(|x|^2+|z|^2)/2}\notag\\
    &\hspace{2.5cm}\times(2\pi)^{(n+1)/2}e^{(|x|^2+|z|^2)/2}\mathbf{1}_{\{|x|^2<1,\,0<|z|^2<t^2\}}\,dz\,dx\notag\\
    =&\int_{|x|<1}dx\int_{0<|z|<t}dz\notag\\
    =&\omega_k\omega_{n+1-k}t^{n+1-k}.\notag
\end{align}
This formula also holds for $k=0$ under the convention
$\omega_0=1$.

If $k=n+1$, then $H_0=0$. Since $f_t(a,b)$ contains the condition
$b>0$, we have
\[
f_t(|G_{n+1}|^2,|H_0|^2)=0,
\]
and therefore
\begin{equation}
    \mathbb{E}f_t(|G_{n+1}|^2,|H_0|^2)=0.\label{eq:right of steiner formula n+1}
\end{equation}

Substituting \eqref{eq:left of steiner formula}, \eqref{eq:right of steiner formula} and \eqref{eq:right of steiner formula n+1} into \eqref{eq:master-steiner} gives
\[
T_C(t)
=
\sum_{k=0}^{n}
\omega_k\omega_{n+1-k}
v_k(C)t^{n+1-k}.
\]

Therefore, if
\[
T_C(t)=\sum_{m=1}^{n+1}c_mt^m,
\]
then
\[
c_m=\omega_{n+1-m}\omega_mv_{n+1-m}(C),\qquad1\leq m\leq n+1.
\]
And \eqref{eq:odd-coefficients} follows from Lemma \ref{lem:conic-gauss-bonnet}.
\end{proof}

\begin{lemma}[The first two nonzero cone coefficients]
\label{lem:cone-bottom-coefficients}
Let $C\subset\R^{n+1}$ be a closed convex cone that is not a
linear subspace, and put
\[
 E=C\cap(-C).
\]
Then $E$ is a linear subspace. Writing $\ell=\dim E$, we have
\[
 v_k(C)=0\quad(0\le k<\ell),\qquad
 v_\ell(C)>0,\qquad v_{\ell+1}(C)>0.
\]
\end{lemma}

\begin{proof}
    The fact that $E$ is a linear subspace follows directly from the definition.
    
	By \cite[Theorem~4.5.1, p.~186,
		equation~(4.105)]{SchneiderCones2022}, a closed convex cone $C$ that is not a
	linear subspace satisfies
	\[
	v_k(C)=0
	\quad\Longleftrightarrow\quad
	k<\dim(C\cap(-C))\ \text{or}\ k>\dim C,
	\qquad 0\le k\le n+1,
	\]
	where $\dim C=\dim\operatorname{span}C$.
	Since $C$ contains $E=C\cap(-C)$ and is not a linear subspace,
	we have $\dim C\ge\ell+1$. Together with the nonnegativity
	of the conic intrinsic volumes, this proves the conclusion.
\end{proof}

\begin{prop}[The curvature bound and its deficit]
\label{prop:nonlinear-bound}
Suppose that $C=\rec(K)$ is not a linear subspace. Then
\begin{equation}\label{eq:intrinsic-cone-values}
 V_M=2v_n(C),\qquad
 L(M)=8\pi\omega_{n-2}v_{n-2}(C).
\end{equation}
Moreover,
\begin{equation}\label{eq:exact-deficit}
 4\pi\omega_{n-2}(1-V_M)-L(M)
 =8\pi\omega_{n-2}
 \sum_{\substack{0\le k\le n-4\\k\equiv n\pmod2}}v_k(C),
\end{equation}
where the sum is empty when $n=3$. In particular,
\begin{equation}
    L(M)\le4\pi\omega_{n-2}(1-V_M).\label{eq:lm upper bound}
\end{equation}
\end{prop}

\begin{proof}
Corollary~\ref{cor:intrinsic-limits} gives
$V_M=c_1/\omega_n$ and $L(M)=6c_3$.
Lemma~\ref{lem:cone-coefficients}, together with
$\omega_1=2$ and $\omega_3=4\pi/3$, gives
\eqref{eq:intrinsic-cone-values}.
Since $C$ is not a linear subspace, the conic Gauss--Bonnet
formula gives
\[
 v_n(C)+v_{n-2}(C)
 +\sum_{\substack{0\le k\le n-4\\k\equiv n\pmod2}}v_k(C)
 =\frac12.
\]
Substituting \eqref{eq:intrinsic-cone-values} proves
\eqref{eq:exact-deficit}.

The nonnegativity of the conic intrinsic volumes gives the stated
bound.
\end{proof}

\section{Product decomposition and linear recession cones}\label{sec:decomposition-linear}

We first give a common product decomposition of $K$.
It will allow us to state the equality criterion entirely in
terms of $K$.

\begin{lemma}[Product decomposition]\label{lem:product}
	Under the hypotheses of Theorem~\ref{thm:main}, put
	$C=\rec(K)$, $E=C\cap(-C)$, and $k=\dim E$. Then $0\le k\le n$,
	and, in the orthogonal decomposition $\R^{n+1}=E^\perp\times E$,
	\begin{equation}\label{eq:product}
		K=D_0\times E,\qquad M=\Sigma\times E,
		\qquad D_0=K\cap E^\perp,\quad\Sigma=\partial_{E^\perp}D_0.
	\end{equation}
	The set $D_0$ is closed and convex, has nonempty interior in
	$E^\perp$, and contains no complete affine line. The hypersurface
	$\Sigma$ in $E^\perp$ is smooth, connected, and complete
	(possibly zero-dimensional), and $M$ has the product metric. Moreover,
	\[
	\rec(K)=\rec(D_0)\times E,
	\]
	and $D_0$ is compact if and only if $C$ is a linear subspace.
	In that case, $E=C$ and $1\le k\le n-1$.
\end{lemma}

\begin{proof}
	The fact that $E$ is a linear subspace follows from the definition. For every $v\in E$, both $v$
	and $-v$ are recession directions, so $K+v=K$. Subtracting the
	$E$-component of each point of $K$ therefore gives
	$K=D_0\times E$. The set $D_0$ is closed and convex. Projecting
	an open ball contained in $K$ onto $E^\perp$ shows that $D_0$
	has nonempty interior in $E^\perp$. The definition of a recession
	direction now gives
	\[
	\rec(K)=\rec(D_0)\times E,
	\qquad \rec(D_0)=C\cap E^\perp.
	\]
	
	To see that $D_0$ contains no complete affine line, suppose that
	$a+\R v\subset D_0$ for some nonzero $v\in E^\perp$. For any
	$z\in D_0$, $t\ge0$, and $s>t$, convexity gives
	\[
	(1-t/s)z+(t/s)(a\pm sv)\in D_0.
	\]
	Letting $s\to\infty$ shows that $z\pm tv\in D_0$. Thus
	$\pm v\in\rec(D_0)$, which implies
	$v\in E\cap E^\perp=\{0\}$, a contradiction.
	
	If $D_0$ is compact, then $\rec(D_0)=\{0\}$, and hence $C=E$ is a linear subspace.
	Conversely, if $D_0$ is unbounded, choose $x_0\in D_0$
	and a sequence $x_j\in D_0$ with $|x_j-x_0|\to\infty$.
	After passing to a subsequence,
	$(x_j-x_0)/|x_j-x_0|\to u$ with $|u|=1$. Convexity and
	closedness give $x_0+tu\in D_0$ for every $t\ge0$.
	By translating $x_0$ to the origin, Lemma~\ref{lem:recession} then shows
	that $u\in\rec(D_0)$. Since $D_0$ contains no complete affine line,
	$-u\notin\rec(D_0)$. Thus $C=\rec(D_0)\times E$ is not a linear subspace. 
	
	Taking boundaries in the product decomposition gives
	$M=\Sigma\times E$. The normal to $M$ is perpendicular to $E$,
	so the implicit function theorem shows that the section
	$\Sigma=M\cap E^\perp$ is smooth. Orthogonality gives the
	product metric; connectedness and completeness of $\Sigma$
	follow from those of $M$.
	The case $k=n+1$ would give $K=\R^{n+1}$, so $k\le n$.
	If $D_0$ is compact, $k=0$ would make $M$ compact, while $k=n$
	would make $D_0$ a nondegenerate compact interval and $M$
	disconnected. Hence $1\le k\le n-1$ in the compact case.
\end{proof}

Suppose that $C=\rec(K)=L$ is a $k$-dimensional linear subspace.
Lemma~\ref{lem:product} gives the decomposition
\[
 K=D\times L,\qquad M=\Sigma\times L,\qquad\Sigma=\partial D,
\]
where $D\subset L^\perp$ is compact with nonempty interior,
$\Sigma$ is smooth and connected, and $1\le k\le n-1$.
The metric on $M$ is the product metric. We now compute the
curvature integral in this case.

\begin{prop}[The integral in the product case]\label{prop:product-limit}
With the notation above, $V_M=0$ and
\begin{equation}\label{eq:product-limit}
 \lim_{r\to\infty}r^{2-n}\int_{B_p^M(r)}R\dd V
 =\begin{cases}
 8\pi\omega_{n-2},&k=n-2,\\
 0,&k\ne n-2.
 \end{cases}
\end{equation}
\end{prop}

\begin{proof}
Identify $L$ with $\R^k$ and take $p=(p_\Sigma,0)$. The scalar
curvature of the product is $R_\Sigma$. The product-distance
formula and Fubini's theorem yield
\begin{equation}\label{eq:fubini}
 \int_{B_p^M(r)}R\dd V
 =\omega_k\int_\Sigma R_\Sigma(x)
 \bigl(r^2-d_\Sigma(p_\Sigma,x)^2\bigr)_+^{k/2}\dd V_\Sigma(x),
\end{equation}
where $s_+=\max\{s,0\}$. Compactness of $\Sigma$ and dominated
convergence give
\begin{equation}\label{eq:product-leading}
 \lim_{r\to\infty}r^{-k}\int_{B_p^M(r)}R\dd V
 =\omega_k\int_\Sigma R_\Sigma\dd V_\Sigma.
\end{equation}
Also,
\[
 \Vol_M(B_p^M(r))\le\Vol(\Sigma)\omega_k r^k,
\]
so $V_M=0$ because $k<n$.

If $k\le n-3$, \eqref{eq:product-leading} gives an $O(r^k)$
curvature integral, and $k+2-n<0$ proves the zero limit in
\eqref{eq:product-limit}. If $k=n-1$, the factor $\Sigma$ is
one-dimensional, so $R_\Sigma=0$ and the same conclusion holds.
If $k=n-2$, then $D$ is a three-dimensional convex body and
$\Sigma=\partial D$ is a smooth closed convex surface. Thus the Gauss--Bonnet theorem gives
\[
 \int_\Sigma R_\Sigma\dd V_\Sigma
 =2\int_\Sigma K_\Sigma\dd V_\Sigma
 =4\pi\chi(\Sigma)=8\pi.
\]
Now \eqref{eq:product-leading} proves \eqref{eq:product-limit}.
\end{proof}

\section{Completion of main theorem and examples}
\label{sec:equality}

\begin{prop}[Equality case]\label{prop:equality-geometric}
Under the hypotheses of Theorem~\ref{thm:main}, the equality
\begin{equation}\label{eq:equality-criterion}
 L(M)=4\pi\omega_{n-2}(1-V_M)
\end{equation}
holds precisely in one of the following cases:
\begin{enumerate}
\item
$(M,g)$ is isometric to $\mathbb{R}^n$.

\item
After an ambient rigid motion,
\begin{equation}
    M=\partial D_0\times\mathbb{R}^k\subset\mathbb{R}^{n+1},\label{eq:geometric-equality}
\end{equation}
where $n-3\leq k\leq n-2$ and $D_0\subset\mathbb{R}^{n+1-k}$ is a closed noncompact convex set with nonempty interior that contains no complete affine line.
\end{enumerate}
\end{prop}

\begin{proof} Suppose first the equality \eqref{eq:equality-criterion} holds. If $M$ is flat, then its scalar curvature vanishes and thus $L(M)=0$. So $V_M=1$. The rigidity part of Bishop-Gromov's volume comparison theorem then implies $(M,g)$ is isometric to $\mathbb{R}^n$. 

If $M$ is not flat, Lemma~\ref{lem:product} gives $K=D_0\times E$, where $D_0$ is closed convex, has nonempty interior in $E^\perp$, and contains no complete affine line. Let $k=\mathrm{dim}E$. If $D_0$ is
compact, Proposition~\ref{prop:product-limit} gives $V_M=0$ and
$L(M)\in\{0,8\pi\omega_{n-2}\}$. Neither value equals
$4\pi\omega_{n-2}$, so equality is impossible.

Suppose that $D_0$ is noncompact. Lemma~\ref{lem:product} then
shows that $C=\rec(K)$ is not a linear subspace and $k\le n$. If $k=n$ or $k=n-1$, then $\partial D_0$ is $0$ or $1$ dimensional respectively. Hence $M$ is flat, which contradicts our hypothesis. So $k\leq n-2$.

For $n=3$, $k\ge n-3=0$ is automatic. For $n\geq 4$, suppose $0\le k\le n-4$. Choose
$j\in\{k,k+1\}$ with the same parity as $n$. Since
$n$ and $n-3$ has the opposite parity, $0\le j\le n-4$.
Lemma~\ref{lem:cone-bottom-coefficients} gives $v_j(C)>0$,
so the right hand side in \eqref{eq:exact-deficit} is strictly positive, which contradicts \eqref{eq:equality-criterion}. So $k\geq n-3$.

Conversely, if $(M,g)$ is isometric to $\mathbb{R}^n$, then \eqref{eq:equality-criterion} holds evidently.

If \eqref{eq:geometric-equality} holds with the stated properties, the absence of a complete affine line
in $D_0$ implies
$\rec(D_0)\cap(-\rec(D_0))=\{0\}$. Hence
$C\cap(-C)=\{0\}\times\R^k$. So Lemma~\ref{lem:cone-bottom-coefficients} then tells us $v_j(C)=0$ for all $j < k$. Since $n-3\leq k\leq n-2$, we have $v_j(C)=0$ for all $j\leq n-4$. Thus the right hand side of \eqref{eq:exact-deficit} vanishes. This proves the equality.
\end{proof}

\begin{proof}[Proof of Theorem~\ref{thm:main}]
Existence, finiteness, and basepoint independence follow from
Corollary~\ref{cor:intrinsic-limits}.
If $C$ is not a linear subspace,
Proposition~\ref{prop:nonlinear-bound} proves the $4\pi\omega_{n-2}(1-V_M)$ bound.
If $C$ is a linear subspace, Lemma~\ref{lem:product} and
Proposition~\ref{prop:product-limit} show that the limit is zero
when $\dim C\neq n-2$, so the same bound holds unless $K$ and $M$ have the stated form
in \eqref{eq:exception-product}. Those products have the values
in \eqref{eq:exception-values}. Proposition~\ref{prop:equality-geometric}
gives the equality case in the theorem.
\end{proof}

The following examples illustrate strict inequality and equality
in \eqref{eq:strong-bound}. In dimensions $n\ge4$, both
behaviors occur for every volume ratio $V_M\in(0,1)$.

\begin{example}[Strict inequality case]\label{ex:strict-inequality}
Fix $n\ge3$ and $a>0$, and consider the rotationally symmetric hypersurface
\[
M_a=\left\{(x,z)\in \mathbb{R}^n\times\mathbb{R}:z=a\sqrt{1+|x|^2}\right\}\subset \mathbb{R}^{n+1}.
\]
Since
\[
D^2\!\left(a\sqrt{1+|x|^2}\right)>0,
\]
the hypersurface $M_a$ is smooth, complete, and strictly convex. We show that, for every $n\geq4$, it satisfies the strict inequality
\[
L(M_a)<4\pi\omega_{n-2}\bigl(1-V_{M_a}\bigr).
\]

Write
\[
\rho=|x|,
\qquad
\beta=\frac{1}{\sqrt{1+a^2}}\in(0,1).
\]
The induced metric on $M_a$ is
\[
g=\left(1+\frac{a^2\rho^2}{1+\rho^2}\right)d\rho^2+\rho^2 g_{S^{n-1}}.
\]
If $s=s(\rho)$ denotes the intrinsic radial distance from the tip
$p=(0,a)$, then
\[
s(\rho)=\int_0^\rho\sqrt{1+\frac{a^2t^2}{1+t^2}}\,dt,
\]
and hence
\[
\lim_{\rho\to\infty}\frac{\rho}{s(\rho)}=\beta.
\]
Consequently,
\begin{equation}
    V_{M_a}=\beta^{n-1}.\label{eq:avr of ma}
\end{equation}

The principal curvatures of $M_a$ consist of one radial principal
curvature
\[
\kappa_r=\frac{a}{\bigl(1+(1+a^2)\rho^2\bigr)^{3/2}},
\]
and a tangential principal curvature of multiplicity $n-1$,
\[
\kappa_t=\frac{a}{\bigl(1+(1+a^2)\rho^2\bigr)^{1/2}}.
\]
Therefore, by the Gauss equation,
\[
R=2(n-1)\kappa_r\kappa_t+(n-1)(n-2)\kappa_t^2.
\]
Since $\rho\sim\beta s$ as $s\to\infty$, we have
\begin{equation}
    R(s)=\frac{(n-1)(n-2)a^2}{s^2}+o(s^{-2}).\label{eq:scalar of ma}
\end{equation}
Moreover, in the radial coordinate $s$,
\[
dV_g=\rho(s)^{n-1}\,ds\,dV_{S^{n-1}},
\]
and thus
\[
\rho(s)^{n-1}=\beta^{\,n-1}s^{n-1}+o(s^{n-1}).
\]
It follows from \eqref{eq:scalar of ma} that
\[
\begin{aligned}
L(M_a)&:=\lim_{r\to\infty}r^{2-n}\int_{B_p^M(r)}R\,dV_g \\
&=(n-1)|S^{n-1}|a^2\beta^{\,n-1}.
\end{aligned}
\]
Using
\[
|S^{n-1}|=2\pi\omega_{n-2},
\qquad
a^2=\frac{1-\beta^2}{\beta^2},
\]
we obtain 
\[
L(M_a)=2\pi(n-1)\omega_{n-2}\beta^{\,n-3}(1-\beta^2).
\]

On the other hand, by \eqref{eq:avr of ma},
\[
4\pi\omega_{n-2}
\bigl(1-V_{M_a}\bigr)
=
4\pi\omega_{n-2}
(1-\beta^{n-1}).
\]
Thus the desired strict inequality is equivalent to
\begin{equation}
    (n-1)\beta^{n-3}(1-\beta^2)<2(1-\beta^{n-1}).
\end{equation}
For $n\ge4$, this follows from
\begin{align*}
    \frac{1-\beta^{n-1}}{n-1}=\int_\beta^1 t^{n-2}\,dt > \beta^{n-3}\int_\beta^1 t\,dt=\frac{\beta^{n-3}(1-\beta^2)}{2}.
\end{align*}
Hence
\[
L(M_a)<4\pi\omega_{n-2}\bigl(1-V_{M_a}\bigr),\qquad n\geq4.
\]

\end{example}

\begin{example}[Equality case]
\label{ex:positive-equality}
Fix $n\ge3$ and $a>0$, and define
\[
 N_a=\{(x,z)\in\R^3\times\R:z=a\sqrt{1+|x|^2}\},
 \qquad M_a=N_a\times\R^{n-3}.
\]
Put $\beta=(1+a^2)^{-1/2}$ as above. Same calculation gives
\[
 V_{M_a}=\beta^2,\qquad
 L(M_a)=4\pi\omega_{n-2}(1-\beta^2).
\]
\end{example}

\section{General codimension}

We now consider embedded orientable convex submanifolds of arbitrary codimension.

\begin{definition}
Let
\[
f:(M^m,g)\longrightarrow \mathbb{R}^{m+q},
\qquad m\geq 2,\quad q\geq 1,
\]
be a smooth isometric embedding of a complete, connected, orientable Riemannian manifold without boundary. Denote its second fundamental form by $B$.

We say that the embedding is \emph{convex} if, for every $x\in M$ and every normal vector
$\xi\in N_xM$, the quadratic form
\[
X\longmapsto \langle B_x(X,X),\xi\rangle,
\qquad X\in T_xM,
\]
is semidefinite, with its sign allowed to depend on $x$ and $\xi$.
\end{definition}

The following is the form of Jonker's structure theorem that will be used here.

\begin{theorem}[{\cite[pp.~610--611]{Jonker}}]
\label{thm:jonker-embedded}
Let
\[
f:(M^m,g)\longrightarrow \mathbb{R}^{m+q},
\qquad m\geq 2,\quad q\geq 1,
\]
be a convex isometric embedding of a complete, connected, orientable Riemannian manifold without boundary. Then at least one of the following alternatives holds.

\begin{enumerate}
\item[(i)]
There exist an $(m+1)$-dimensional affine subspace
$P\subset\mathbb{R}^{m+q}$ and a nonempty open convex set
$\Omega\subset P$ such that
\[
f(M)=\partial_P\Omega,
\]
where $\partial_P$ denotes the boundary relative to $P$.

\item[(ii)]
There exist a complete, connected, one-dimensional Riemannian manifold
$(\Gamma^1,h)$ without boundary, an isometric embedding
\[
\gamma:(\Gamma^1,h)\longrightarrow\mathbb{R}^{q+1},
\]
a Riemannian isometry
\[
\Phi:
(\Gamma^1\times\mathbb{R}^{m-1},
 h\oplus g_{\mathrm{Euc}})
\longrightarrow (M,g),
\]
and an ambient rigid motion $T$ such that, under the orthogonal identification
\[
\mathbb{R}^{m+q}
=
\mathbb{R}^{q+1}\times\mathbb{R}^{m-1},
\]
one has
\[
(T\circ f\circ\Phi)(s,z)
=
\bigl(\gamma(s),z\bigr).
\]
\end{enumerate}

In alternative \emph{(ii)}, the product metric is flat. Since $f$ is assumed to be an embedding, the curve $\gamma$ is embedded as well.
\end{theorem}

\begin{cor}
\label{cor:general-codimension}
Let
\[
f:(M^m,g)\longrightarrow\mathbb{R}^{m+q},
\qquad m\geq 3,
\]
be a convex isometric embedding of a complete, connected, orientable, noncompact Riemannian manifold without boundary. For $p\in M$, set
\[
V_M
:=
\lim_{r\to\infty}
\frac{\operatorname{Vol}_g(B^M_p(r))}
{\omega_m r^m}.
\]
Then the finite limit
\[
L(M)
:=
\lim_{r\to\infty}
r^{2-m}
\int_{B^M_p(r)}
\operatorname{Scal}_g\,dV_g
\]
exists and is independent of $p$. Moreover,
\begin{equation}
\label{eq:general-codim-bound}
L(M)
\leq
4\pi\omega_{m-2}(1-V_M),
\end{equation}
except when, after an ambient rigid motion,
\begin{equation}
\label{eq:general-codim-exceptional}
f(M)
=
\partial D\times\mathbb{R}^{m-2}
\subset
\mathbb{R}^3\times\mathbb{R}^{m-2}
=
\mathbb{R}^{m+1}
\subset
\mathbb{R}^{m+q},
\end{equation}
where $D\subset\mathbb{R}^3$ is a compact convex set with nonempty interior and smooth boundary. In this exceptional case,
\[
V_M=0,
\qquad
L(M)=8\pi\omega_{m-2}.
\]

Equality in \eqref{eq:general-codim-bound} holds precisely in one of the following cases:
\begin{enumerate}
\item
$(M,g)$ is isometric to $\mathbb{R}^m$.

\item
After an ambient rigid motion,
\[
f(M)
=
\partial D_0\times\mathbb{R}^k
\subset
\mathbb{R}^{m+1}
\subset
\mathbb{R}^{m+q},
\]
where
\[
m-3\leq k\leq m-2
\]
and
\[
D_0\subset\mathbb{R}^{m+1-k}
\]
is a closed, noncompact convex set with nonempty interior that contains no complete affine line.
\end{enumerate}
\end{cor}

\begin{proof}
For every orthonormal pair $X,Y\in T_xM$, choose an orthonormal basis
$\{\xi_\alpha\}$ of the normal space $N_xM$. The Gauss equation gives
\[
K_M(X,Y)
=
\sum_\alpha
\left(
\langle B(X,X),\xi_\alpha\rangle
\langle B(Y,Y),\xi_\alpha\rangle
-
\langle B(X,Y),\xi_\alpha\rangle^2
\right).
\]
For each $\alpha$, the symmetric bilinear form
\[
(X,Y)\longmapsto \langle B(X,Y),\xi_\alpha\rangle
\]
is semidefinite by convexity. Therefore $K_M\geq 0$.
In particular, $\operatorname{Ric}_g\geq0$, and Bishop--Gromov volume comparison yields $V_M\in[0,1]$.

Suppose first that $M$ is flat. Then
\[
\operatorname{Scal}_g\equiv0,\qquad L(M)=0,
\]
so \eqref{eq:general-codim-bound} follows immediately. Equality holds if and only if
$V_M=1$. By the equality case of Bishop--Gromov comparison, this is equivalent to
\[
(M,g)\cong\mathbb{R}^m.
\]

Now suppose that $M$ is not flat. Alternative (ii) of
Theorem~\ref{thm:jonker-embedded} is impossible, hence alternative (i) holds. Thus there exist an affine subspace $P\cong\mathbb{R}^{m+1}$ and a nonempty open convex set $\Omega\subset P$ such that $f(M)=\partial_P\Omega$. 

Since $f$ is an isometric embedding, $f(M)$ is a smooth, connected, complete, noncompact convex hypersurface in $P$. Set
\[
K=\overline{\Omega}\subset P.
\]
Then $K$ is a closed convex set with nonempty interior and $\partial_P K=f(M)$. The embedding identifies the intrinsic metric on $\partial_P K$ with $g$.

Applying Theorem~\ref{thm:main} in the affine Euclidean space
$P\cong\mathbb{R}^{m+1}$, with $n=m$, gives the existence, finiteness, and basepoint independence of $L(M)$, as well as
\[
L(M)
\leq
4\pi\omega_{m-2}(1-V_M),
\]
except for the case
\[
M=\partial D\times\mathbb{R}^{m-2},
\]
where $D\subset\mathbb{R}^3$ is a compact convex body with smooth boundary. In this exceptional case, Theorem~\ref{thm:main} gives
\[
V_M=0,
\qquad
L(M)=8\pi\omega_{m-2}.
\]

And the equality case is the same as Theorem~\ref{thm:main}.
\end{proof}


\begin{bibdiv}
\begin{biblist}

\bib{AmelunxenLotz2017}{article}{
  author  = {Amelunxen, Dennis},
  author  = {Lotz, Martin},
  title   = {Intrinsic Volumes of Polyhedral Cones: A Combinatorial Perspective},
  journal = {Discrete \& Computational Geometry},
  volume  = {58},
  number  = {2},
  pages   = {371--409},
  year    = {2017},
  doi     = {10.1007/s00454-017-9904-9}
}

\bib{ALMT}{article}{
  author  = {Amelunxen, Dennis},
  author  = {Lotz, Martin},
  author  = {McCoy, Michael B.},
  author  = {Tropp, Joel A.},
  title   = {Living on the Edge: Phase Transitions in Convex Programs with Random Data},
  journal = {Information and Inference: A Journal of the IMA},
  volume  = {3},
  number  = {3},
  pages   = {224--294},
  year    = {2014},
  doi     = {10.1093/imaiai/iau005}
}

\bib{Cheng2026}{article}{
  author  = {Cheng, Haoxuan},
  title   = {Unbounded Normalized Scalar Curvature Integrals in Dimension Four},
  year    = {2026},
  eprint  = {arXiv:2609.10323}
}

\bib{CohnVossen1935}{article}{
  author  = {Cohn-Vossen, Stefan},
  title   = {K{\"u}rzeste Wege und Totalkr{\"u}mmung auf Fl{\"a}chen},
  journal = {Compositio Mathematica},
  volume  = {2},
  pages   = {69--133},
  year    = {1935}
}

\bib{Gomes1983}{article}{
  author  = {de Miranda Gomes, J.},
  title   = {On Isometric Immersions with Semi-Definite Second Quadratic Forms},
  journal = {Anais da Academia Brasileira de Ci\^encias},
  volume  = {55},
  number  = {2},
  pages   = {145--146},
  year    = {1983}
}

\bib{Deng2026}{article}{
  author  = {Deng, Jialong},
  title   = {Cohn--Vossen-Type Inequalities for Three-Manifolds and Locally Conformally Flat Manifolds},
  year    = {2026},
  eprint  = {arXiv:2606.01368}
}

\bib{HaoZhu2026}{article}{
  author  = {Hao, Tianze},
  author  = {Zhu, Jintian},
  title   = {A Counterexample to Yau's Conjectured Asymptotic Scalar-Curvature Integral Bound},
  year    = {2026},
  eprint  = {arXiv:2609.06533}
}

\bib{Hartman1964}{article}{
  author  = {Hartman, Philip},
  title   = {Geodesic Parallel Coordinates in the Large},
  journal = {American Journal of Mathematics},
  volume  = {86},
  number  = {4},
  pages   = {705--727},
  year    = {1964},
  doi     = {10.2307/2373154}
}

\bib{Hu2026}{article}{
  author  = {Hu, Qixuan},
  title   = {Scalar Curvature Growth on Nonnegatively Curved Three-Manifolds},
  year    = {2026},
  eprint  = {arXiv:2609.26717},
  url     = {https://arxiv.org/abs/2609.26717}
}

\bib{Jonker}{article}{
  author  = {Jonker, L. B.},
  title   = {Immersions with Semi-Definite Second Fundamental Forms},
  journal = {Canadian Journal of Mathematics},
  volume  = {27},
  number  = {3},
  pages   = {610--617},
  year    = {1975},
  doi     = {10.4153/CJM-1975-071-9}
}

\bib{Ma2026}{article}{
  author  = {Ma, Shiguang},
  title   = {Cohn--Vossen Theory for Locally Conformally Flat Manifolds},
  year    = {2026},
  eprint  = {arXiv:2606.13979}
}

\bib{McCoyTropp}{article}{
  author  = {McCoy, Michael B.},
  author  = {Tropp, Joel A.},
  title   = {From Steiner Formulas for Cones to Concentration of Intrinsic Volumes},
  journal = {Discrete \& Computational Geometry},
  volume  = {51},
  number  = {4},
  pages   = {926--963},
  year    = {2014},
  doi     = {10.1007/s00454-014-9595-4}
}

\bib{MullerSverak1995}{article}{
  author  = {M\"uller, Stefan},
  author  = {\v{S}ver\'ak, Vladim\'ir},
  title   = {On Surfaces of Finite Total Curvature},
  journal = {Journal of Differential Geometry},
  volume  = {42},
  number  = {2},
  pages   = {229--258},
  year    = {1995},
  doi     = {10.4310/jdg/1214457233}
}

\bib{MunteanuWang}{article}{
  author  = {Munteanu, Ovidiu},
  author  = {Wang, Jiaping},
  title   = {Sharp Integral Bound of Scalar Curvature on $3$-Manifolds},
  year    = {2025},
  eprint  = {arXiv:2505.10520}
}

\bib{Petersen}{book}{
  author    = {Petersen, Peter},
  title     = {Riemannian Geometry},
  edition   = {3},
  series    = {Graduate Texts in Mathematics},
  volume    = {171},
  publisher = {Springer},
  address   = {Cham},
  year      = {2016}
}

\bib{Sacksteder1960}{article}{
  author  = {Sacksteder, Richard},
  title   = {On Hypersurfaces with No Negative Sectional Curvatures},
  journal = {American Journal of Mathematics},
  volume  = {82},
  number  = {3},
  pages   = {609--630},
  year    = {1960},
  doi     = {10.2307/2372973}
}

\bib{SchneiderCones2022}{book}{
  author    = {Schneider, Rolf},
  title     = {Convex Cones: Geometry and Probability},
  series    = {Lecture Notes in Mathematics},
  volume    = {2319},
  publisher = {Springer},
  address   = {Cham},
  year      = {2022},
  doi       = {10.1007/978-3-031-15127-9}
}

\bib{Xu2020}{article}{
  author  = {Xu, Guoyi},
  title   = {Integral of Scalar Curvature on Non-Parabolic Manifolds},
  journal = {Journal of Geometric Analysis},
  volume  = {30},
  number  = {1},
  pages   = {901--909},
  year    = {2020},
  doi     = {10.1007/s12220-019-00174-7}
}

\bib{Xu2024}{article}{
  author  = {Xu, Guoyi},
  title   = {Integral of Scalar Curvature on Manifolds with a Pole},
  journal = {Proceedings of the American Mathematical Society},
  volume  = {152},
  number  = {11},
  pages   = {4865--4872},
  year    = {2024},
  doi     = {10.1090/proc/16584}
}

\bib{XuCollapse2026}{article}{
  author  = {Xu, Guoyi},
  title   = {Integral of Scalar Curvature under Volume Collapse},
  year    = {2026},
  eprint  = {arXiv:2609.10160}
}

\bib{Xu2026Sharp}{article}{
  author  = {Xu, Guoyi},
  title   = {Integral of Scalar Curvature: Sharp Asymptotic Bounds and Rigidity},
  year    = {2026},
  eprint  = {arXiv:2609.26357}
}

\bib{Yau1992}{incollection}{
  author    = {Yau, Shing-Tung},
  title     = {Open Problems in Geometry},
  booktitle = {Shiing-Shen Chern---A Great Geometer of the Twentieth Century},
  editor    = {Yau, Shing-Tung},
  publisher = {International Press},
  address   = {Hong Kong},
  pages     = {275--319},
  year      = {1992}
}
\end{biblist}
\end{bibdiv}
\end{document}